\documentclass{article}
\usepackage[english]{babel}
\usepackage{a4}
\usepackage{amsmath}
\usepackage{graphicx}
\usepackage{amsfonts}
\usepackage{amssymb}
\usepackage{amsthm}
\usepackage{enumitem}
\setlist{nosep}
\setlist[enumerate,1]{label=(\arabic*)}

\usepackage{xcolor, soul}
\sethlcolor{green}

\newtheorem{claim}{Claim}

\newcommand*{\claimproof}{Proof of the Claim}
\newenvironment{proofcl}[1][\claimproof]{\begin{proof}[#1]
}{\end{proof}}

\newtheorem{theorem}{Theorem}
\newtheorem{proposition}[theorem]{Proposition}

\newtheorem{conjecture}[theorem]{Conjecture}

\newtheorem{originalConjecture}{Conjecture}

\newcommand{\Aa}{\mathcal{A}}
\newcommand{\Bb}{\mathcal{B}}
\newcommand{\Cc}{\mathcal{C}}
\newcommand{\Dd}{\mathcal{D}}
\newcommand{\Ee}{\mathcal{E}}
\newcommand{\Ff}{\mathcal{F}}
\newcommand{\Gg}{\mathcal{G}}
\newcommand{\Oo}{\mathcal{O}}
\newcommand{\Pp}{\mathcal{P}}
\newcommand{\Qq}{\mathcal{Q}}
\newcommand{\Rr}{\mathcal{R}}

\newcommand{\sm}{\setminus}

\newcommand{\splitatcommas}[1]{%
  \begingroup
  \begingroup\lccode`~=`, \lowercase{\endgroup
    \edef~{\mathchar\the\mathcode`, \penalty0 \noexpand\hspace{0pt plus 1em}}%
  }\mathcode`,="8000 #1%
  \endgroup
}

\title{The number of abundant elements in union-closed families without small sets\thanks{The
work of the first and third author was supported
by project 20-09525S of the Czech Science Foundation.
All authors are affiliated with the Faculty
of Applied Sciences, University of West Bohemia, Czech Republic. 
E-mails: {\tt kabela@kma.zcu.cz, polakmi@students.zcu.cz, teska@kma.zcu.cz}.}}
\author{Adam Kabela \and Michal Pol\'{a}k \and Jakub Teska}
\date{} 

\begin{document}
\maketitle

\begin{abstract}
We let $\Ff$ be a finite family of sets closed under taking unions and $\emptyset \not \in \Ff$,
and call an element \emph{abundant}
if it belongs to more than half of the sets of~$\Ff$.
In this notation, the classical Frankl's conjecture (1979)
asserts that $\Ff$ has an abundant element.
As possible strengthenings, Poonen (1992) conjectured that if $\Ff$ has 
precisely one abundant element, then
this element belongs to each set of~$\Ff$,
and Cui and Hu (2019) investigated whether
$\Ff$ has at least $k$ abundant elements if a smallest set of $\Ff$ is of size at least $k$.
Cui and Hu conjectured that this holds for $k = 2$ and asked 
whether this also holds for the cases $k = 3$ and $k > \frac{n}{2}$
where $n$ is the size of the largest set of $\Ff$.

We show that $\Ff$ has at least $k$ abundant elements if $k \geq n - 3$,
and that $\Ff$ has at least $k - 1$ abundant elements if $k = n - 4$,
and we construct a union-closed family
with precisely $k - 1$ abundant elements for every
$k$ and $n$ satisfying $n - 4 \geq k \geq 3$ and $n \geq 9$
(and for $k = 3$ and $n = 8$).
We also note that $\Ff$ always has at least $\min \{ n, 2k - n + 1 \}$ abundant elements.
On the other hand, we construct a union-closed family
with precisely two abundant elements for every $k$
and $n$ satisfying $n \geq \max \{ 3, 5k-4 \}$.
Lastly, we show that Cui and Hu's conjecture for $k = 2$
stands between Frankl's conjecture and Poonen's conjecture.
\end{abstract}

\section{Introduction}
\label{introduction}
We recall that a collection of sets is a \emph{family} if all sets are distinct,
and a family $\Ff$ of sets is \emph{union-closed}
if for every pair of sets $A$ and $B$ from $\Ff$,
their union $A \cup B$ also belongs to $\Ff$.
For families which do not contain the empty set,
the classical union-closed sets conjecture states the following.
\begin{conjecture}
\label{frankl}
If $\Ff$ is a finite union-closed family of sets such that $\emptyset \not \in \Ff$,
then some element belongs to more than half of the sets of $\Ff$.
\end{conjecture}
The conjecture is associated with Frankl
and its origins go back to 1970s~\cite{F,BBE}. 
An engaging historical review and numerous related results 
can be found in the survey paper of Bruhn and Schaudt~\cite{BS}.
Notably, Conjecture~\ref{frankl} is also investigated in equivalent formulations
concerning lattices, independent sets in bipartite graphs,
and basis sets in union-closed families (for instance, see~\cite{BS, P, BCST, W}).

Conjecture~\ref{frankl} remains wide open
despite the amount of research on the topic
and the recent breakthrough result of Gilmer~\cite{G}
instantly followed by~\cite{AHS, CL, Pe, S} and~\cite{C}
which proved the statement of Conjecture~\ref{frankl}
in weaker forms with the half replaced by a smaller constant.

In~particular, it is even non-trivial to show that Conjecture~\ref{frankl} is true for small families,
and there is a sequel of papers~\cite{P, LF, Ma, BM, VZ}
showing that Conjecture~\ref{frankl} holds for union-closed families whose largest set is of size at most
7, 9, 10, 11, 12, respectively.
The proofs of the partial results use the fact
that Conjecture~\ref{frankl} is true for families which contain a set of size $1$ or $2$
(first shown in~\cite{SR}).
Hence, the proofs can essentially consider families which contain no set of size $1$ and $2$,
and with more involved arguments also reduce the number of possible sets of size $3$ and $4$
(for instance, see~\cite{Mo}).
This motivates the present study of families whose smallest set is of size at least $k$.
%
%
%
%
%
%

On the other hand, 
it seems that we are far from disproving Conjecture~\ref{frankl}
since there are stronger conjectures which are still open.
As remarked by Poonen~\cite{P},
it seems that union-closed families usually contain many elements
satisfying the statement of Conjecture~\ref{frankl}.
Considering families with only one such element,
Poonen~\cite{P} stated the following conjectures.

\begin{conjecture}
\label{poonen3}
Let $\Ff$ be a finite union-closed family of sets such that $\emptyset \not \in \Ff$.
If precisely one element belongs to more than half of the sets of $\Ff$,
then this element belongs to each set of $\Ff$.
\end{conjecture}

Given a family $\Ff$ of sets,
we say that elements $a$ and $b$ are \emph{twins} if each set $A$ of $\Ff$ satisfies that
$|A \cap \{a,b\}| \neq 1$, and
we say that $\Ff$ is \emph{twin-free} if $\Ff$ contains no twins.

\begin{conjecture}
\label{poonen4}
Let $\Ff$ be a finite union-closed twin-free family of sets such that $\emptyset \not \in \Ff$,
and let $M$ be the largest set of $\Ff$.
If precisely one element belongs to more than half of the sets of $\Ff$,
then $\Ff$ is precisely the family consisting of all sets of $2^M$ which contain this element.
\end{conjecture}

It appears that Conjectures~\ref{poonen3} and~\ref{poonen4} are still open and little is known
in terms of partial results.
The phenomenon of many elements satisfying the statement of Conjecture~\ref{frankl}
is also addressed in a recent work of Cui and Hu~\cite{CH}.
As a possible strengthening of Conjecture~\ref{frankl},
Cui and Hu investigated whether a finite union-closed family of sets
has at least $k$ such elements if its smallest set is of size at least~$k$.
In particular, the case $k = 1$ is precisely Conjecture~\ref{frankl},
and they asked whether it holds for the cases $k = 3$ and $k > \frac{n}{2}$
(where $n$ is the size of the largest set),
and conjectured that it holds for the case $k = 2$.

\begin{conjecture}
\label{CuiHuS2}
If $\Ff$ is a finite union-closed family of sets whose smallest set is of size at least $2$,
then there are at least two elements such that each belongs to more than half of the sets of $\Ff$.
\end{conjecture}


In the present paper, we show that if a smallest set is large, then
many elements satisfy the statement of Conjecture~\ref{frankl} as follows.
\begin{theorem}
\label{abundance}
Let $\Ff$ be a union-closed family of sets
and $k$ and $n$ be the sizes of its smallest and largest set, respectively.
Let $f$ be the number of elements $x$ such that $x$ belongs to more than half of the sets of $\Ff$.
The following statements are satisfied.
\begin{enumerate}
\item
If $k \geq n - 3$ then $f \geq k$.
\item
If $k = n - 4$ then $f \geq k - 1$.
\item
$f \geq \min \{ n, 2k - n + 1 \}$.
\end{enumerate}
\end{theorem}

The proof of Theorem~\ref{abundance} is given in Section~\ref{conditions}.
We complement the result of Theorem~\ref{abundance}
by constructing families whose smallest set is of size $k$,
but fewer than $k$ elements satisfy the statement of Conjecture~\ref{frankl}.
In particular, we show
that the difference of $k$ and $f$ can be arbitrarily large.
The properties of the present constructions are summarized as follows.

\begin{theorem}
\label{counterexamples}
We say that a union-closed family $\Ff$ of sets is an $(f,k,n)$-construction
if
there are precisely $f$ elements each of which belongs to more than half of the sets of $\Ff$
and $k$ and $n$ are the sizes of a smallest and the largest set of $\Ff$, respectively.
The following statements are satisfied.
\begin{enumerate}
\item
There is a $(2,3,8)$-construction, $(3,4,9)$-construction, $(4,5,9)$-construction and $(5,6,10)$-construction.
The constructions are twin-free.
\item
There is a twin-free $(2,k,n)$-construction for every $k$ and $n$ which satisfy
${\sum_{i=k-1}^{\lfloor \frac{n}{2} \rfloor - 1} \binom{\lfloor \frac{n}{2} \rfloor - 1}{i}
>
\binom{n-3}{k-3} + \binom{\lfloor \frac{n}{2} \rfloor -2}{k-2}}.$
This inequality holds for every $k$ and large enough~$n$.  
\item
There is a $(2,k,n)$-construction for every $k$ and $n$ satisfying ${n \geq \max \{ 3, 5k - 4 \}}$.
Furthermore, $n \geq \max \{ 3, 5k - 8 \}$ suffices for even $k$.
\item
There is a $(k - 1,k,n)$-construction for every $k$ and $n$ satisfying $n - 4 \geq k \geq 3$ and $n \geq 9$.
\end{enumerate}
\end{theorem}

The proof of Theorem~\ref{counterexamples} and details on the constructions
are given in Section~\ref{families}.
Regarding the bounds shown in Theorem~\ref{abundance},
we note that the inequality $f \geq k$ in statement (1) it tight 
(this is easy to see),
and the inequality $f \geq k - 1$ in statement (2) is also tight
(this follows from item (4) of Theorem~\ref{counterexamples}),
and there likely is an ample room for improvement of statement (3) of Theorem~\ref{abundance}.

In~particular,
statements (1) of Theorem~\ref{abundance} and (4) of Theorem~\ref{counterexamples}
resolve the questions of~\cite{CH} for every $k$ greater than $2$.
We leave the case $k = 2$ open and show that  
it stands between Conjectures~\ref{frankl} and~\ref{poonen3} as follows.

\begin{proposition}
\label{hierarchy}
Each of the following implications holds.
\begin{enumerate}
\item
Conjecture~\ref{poonen4} implies Conjecture~\ref{poonen3}.
\item
Conjecture~\ref{poonen3} implies Conjecture~\ref{CuiHuS2}.
\item
Conjecture~\ref{CuiHuS2} implies Conjecture~\ref{frankl}.
\end{enumerate}
\end{proposition}

The proof of Proposition~\ref{hierarchy} is given in Section~\ref{relations}.
We should say that there is a minor difference between the present formulations of
Conjectures~\ref{frankl}, \ref{poonen3}, \ref{poonen4} and~\ref{CuiHuS2}
and formulations of these conjectures in the literature.
For the sake of completeness,
we recall the common formulations in Section~\ref{relations}
and show that each is equivalent to the respective present formulation. 

This minor difference is that
Conjectures~\ref{frankl}, \ref{poonen3}, \ref{poonen4} and~\ref{CuiHuS2}
are commonly formulated for union-closed families possibly containing the empty set
and consider elements which belong to at least half of the sets (with a non-strict inequality).
In particular,
the investigation of~\cite{CH} concerns the size of a smallest non-empty set,
and the original formulations in~\cite{P}
state that the considered element belongs to all non-empty sets.  
However,
it seems more natural to state the conjectures in the present form
for union-closed families which do not contain the empty set,
and accordingly 
we say that an element is \emph{abundant} for a family~$\Ff$
if it belongs to more than half of the sets of~$\Ff$ (we use a strict inequality).  

We also note that the results of Theorems~\ref{abundance} and~\ref{counterexamples}
trivially translate to the original setting of~\cite{CH}.
In particular,
if $\Ff$ is a union-closed family such that $\emptyset \in \Ff$,
then we apply Theorem~\ref{abundance} to the family $\Ff \sm \emptyset$
and obtain $f$ elements which belong to more than half of the sets of $\Ff \sm \emptyset$,
and we conclude that each belongs to at least half of the sets of $\Ff$.
Furthermore,
for each of the triples $(f,k,n)$ considered in Theorem~\ref{counterexamples} distinct from $(4,5,9)$, 
we present an $(f,k,n)$-construction such that
each of the remaining $n-f$ elements belongs to strictly less than half of its sets.
Hence, the constructions are slightly stronger than needed,
and thus relevant in the present setting and also in the original setting of~\cite{CH}.
(For the $(4,5,9)$ triple, we can simply add the empty set to the
$(4,5,9)$-construction used and conclude that
$5$ elements belong to strictly less than half of the sets of the modified construction.)
%
%
\section{Sufficient conditions for many abundant elements}
\label{conditions}
In the present section, we prove Theorem~\ref{abundance}.
We note that statement (3) of Theorem~\ref{abundance} is shown by a simple averaging argument
and statements (1) and (2) are shown by slightly more involved double counting arguments.
For the sake of simplicity, we view the families
from a dual perspective in the proofs of statements (1) and (2).
We start by showing statement (3).
\begin{proof}[Proof of Theorem~\ref{abundance}]
We show statement (3).
Clearly, if $k = n$ then $f = n \geq \min \{ n, 2k - n + 1 \}$ as desired.
Hence, we can assume that $k \leq n - 1$ and we need to show that $f \geq 2k -n + 1$.
We let $m$ be the number of sets in family $\Ff$
and $s = \sum_{A \in \Ff}|A|$.
We note that $s > km$ (since $k \leq n - 1$).
On the other hand, we observe that
$fm + (n-f)\frac{m}{2} \geq s$
(since $n-f$ elements belong to at most half of the sets of $\Ff$).
We combine the inequalities and get
$\frac{n+f}{2} > k$, and it follows that
$f \geq 2k -n + 1$.

Next, we show statement (1).
For the sake of a contradiction,
we suppose that there is a union-closed family $\Ff$
whose largest set, say $M$, is of size $n$ and smallest set is of size $k$
so that $n \leq k + 3$, and that there are at least $n-k+1$ elements of $M$
such that each of these elements belongs to at most $\frac{1}{2}|\Ff|$ sets of $\Ff$
(where $|\Ff|$ is the number of sets in $\Ff$).   
We consider a set of precisely $n-k+1$ such elements, and
we let $X$ denote this set and let $d = n-k$.

We consider the dual family $\Dd$ defined as follows.
A set $A$ belongs to $\Dd$ if and only if the set $M \sm A$ belongs to $\Ff$.
The definition yields that $|\Dd| = |\Ff|$ and $\emptyset \in \Dd$ and a largest set of $\Dd$ is of size $d$
and each element of $X$ belongs to at least $\frac{1}{2}|\Dd|$ sets of~$\Dd$.
Furthermore, we observe that $\Dd$ is closed under taking intersections
(since $\Ff$ is union-closed).

For each set $A$ of $\Dd$, we define its \emph{rank} as $r(A) = |A \cap X|$.
By the assumption on the elements of $X$,
we note that the sum of ranks taken over all sets of $\Dd$ is at least $\frac{d+1}{2}|\Dd|$
(since $|X| = d+1$).
In other words, the average rank is bounded from below as follows.
\begin{equation}
\label{eq:1}
\frac{ \sum_{A \in \Dd} r(A) }{|\Dd|} \geq \frac{d+1}{2}
\end{equation}
We note that $r(A) \leq d$ for every $A$.
For every $i$ of $\{0,1,\dots,d\}$,
we let $r_i$ denote the number of sets of $\Dd$ whose rank is equal to $i$. 
Now, we can rewrite inequality~\eqref{eq:1} and obtain the following.
\begin{equation}
\label{eq:2}
\frac{ \sum_{i=0}^d i \cdot r_i }{\sum_{i=0}^d r_i} \geq \frac{d+1}{2}
\end{equation}
We use inequality~\eqref{eq:2} and the fact that
$r_0 \geq 1$, and we observe the following.
\begin{itemize}
\item
$d \geq 2$.
\item
If $d = 2$ then $r_2 \geq 3$.
\item
If $d = 3$ then $r_3 \geq 2$.
\end{itemize}
On the other hand, we note that $r_d \leq d + 1$ (since $\binom{d+1}{d} = d+1$).
If $r_d = d + 1$, then we observe that $r_1 = d + 1$
(since $\Dd$ contains all possible sets of rank $d$ and $\Dd$ is intersection-closed),
and we use that $d \leq 3$ and obtain a contradiction with inequality~\eqref{eq:2}.

Hence, we can assume that $d = 3$ and $2 \leq r_3 \leq 3$,
and we discuss two cases.
For the case that $r_3 = 2$,
inequality~\eqref{eq:2} implies that $r_1 = 0$.
We consider the two sets of rank $3$ in $\Dd$ and let $I$ denote their intersection
(and note that $|I| = r(I) = 2$). 
We observe that every set of positive rank in $\Dd$ contains $I$ as a subset 
(since $r_1 = 0$ and $\Dd$ is intersection-closed).
It follows that some element of $X$ belongs to only one set of $\Dd$,
a contradiction.
For the case that $r_3 = 3$,
we observe that $r_1 = 1$
(inequality~\eqref{eq:2} implies that $r_1 \leq 1$,
and $r_1 \geq 1$ follows since $\Dd$ is intersection-closed),
and we let $x$ denote the element of $X$ which belongs to the set of rank $1$.
We observe that $x$ belongs to each of the three sets of rank $3$ in $\Dd$
(since $\{x\}$ is the only set of rank $1$ and $\Dd$ is intersection-closed),
and it follows that $x$ also belongs every set of rank $2$ in $\Dd$.
Finally, we consider the elements of $X \sm \{x\}$ and count their occurrences, 
and we conclude that some element of $X \sm \{x\}$ belongs to less than
$\frac{1}{2}|\Dd|$ sets of~$\Dd$, a contradiction.

Lastly, we show statement (2).
Similarly to the proof of item (1),
we define the dual family $\Dd$;
and we note that in the setting of item (2) a largest set of $\Dd$ has size~$4$.
For the sake of a contradiction, we suppose that there are at least 
$6$ elements each of which belongs to at least $\frac{1}{2}|\Dd|$ sets of $\Dd$.

In addition, we say that an element $a$ is \emph{dominated} by an element $b$  
if $b$ belongs to every set of $\Dd$ containing $a$
and $a \neq b$.
We show the following.

\begin{claim} 
\label{c1}
If set $\{a\}$ does not belong to $\Dd$,
then element $a$ is dominated.
\end{claim}
\begin{proofcl}[Proof of Claim~\ref{c1}]
For the sake of a contradiction, we suppose that 
$\{a\}$ does not belong to $\Dd$, but
$a$ is not dominated by any element.
In particular, we have that $a$ belongs to some set of $\Dd$ 
(otherwise $a$ is dominated by every element).
We let $A$ be a smallest set of $\Dd$ containing $a$,
and we note that $|A| \geq 2$.
Since $a$ is not dominated,
we observe that $\Dd$ contains a set $B$
such that $a$ belongs to $B$
and $A$ is not a subset of~$B$.
We note that $a$ belongs to $A \cap B$,
and $A \cap B$ belongs to~$\Dd$
(since $\Dd$ is intersection-closed).
We conclude that $A \cap B$ is smaller than $A$,
which contradicts the choice of~$A$.
\end{proofcl}

We let $H$ denote the set of all elements which belong to at least $\frac{1}{2}|\Dd|$ sets of~$\Dd$.
We call a set $X$ a \emph{crew} 
if $X \subseteq H$ and $|X| = 6$.
%
For every set $A$ of $\Dd$,
we let $r(X,A) = |A \cap X|$.
Furthermore, we let $r_i(X)$ denote the number of sets $A$ of $\Dd$
such that $r(X,A) = i$.
We say that a set $A$ is \emph{relevant} for $X$ if
$A$ belongs to $\Dd$
and
$r(X,A) = 4$
and
$|A \cap \{a,b\}| \in \{0,2\}$ for every pair of elements $a, b$ of $X$ such that
$a$ is dominated by $b$.
We let $\Rr(X)$ be the family of all relevant sets for~$X$,
and we show the following.

\begin{claim} 
\label{c2}
For every crew $X$,
we have $|\Rr(X)| \geq r_2(X) + 2r_1(X) + 3$.
\end{claim}
\begin{proofcl}[Proof of Claim~\ref{c2}]
We consider an arbitrary crew $X$.
The definition of $X$ yields that
\begin{equation*}
\sum_{A \in \Dd} r(X, A) \geq 6 \cdot \frac{1}{2} |\Dd| = 3|\Dd|.
\end{equation*}
In addition, we define
$w(X) = \sum_{A \in \Dd} r(X, A) - 3|\Dd|$ and note that $w(X)$ is non-negative.
In other words,
we consider the elements of $X$ which occur in more than half of the sets of $\Dd$
and let $w(X)$ account for the total number of these additional occurrences.  
We observe that
\begin{equation*}
\begin{split}
w(X) &= \sum_{A \in \Dd} r(X, A) - 3|\Dd|
	= \sum_{A \in \Dd} \left(r(X, A) - 3 \right)
	= \sum_{i=0}^4 (i-3)r_i(X)\\
	&= r_4(X) - r_2(X) - 2r_1(X) - 3r_0(X)
\end{split}
\end{equation*}
where the third equality follows from the fact that a largest set of $\Dd$ is of size~$4$.
We use that $r_0(X) \geq 1$ and obtain
\begin{equation*}
r_4(X) - w(X) \geq r_2(X) + 2r_1(X) + 3.
\end{equation*}

It remains to show that 
$|\Rr(X)| \geq r_4(X) - w(X)$.
We rewrite the inequality as  
$w(X) \geq r_4(X) - |\Rr(X)|$
and note that the right-hand side accounts for the number of sets $A$ of $\Dd$
such that $r(X, A) = 4$ and $A$ does not belong to $\Rr(X)$.
It suffices to show that each such set contributes at least $1$ to $w(X)$.
To this end, we consider an arbitrary set $A$ of $\Dd \sm \Rr(X)$ satisfying $r(X, A) = 4$.
The definition of $\Rr(X)$ implies that
there is pair of elements $a, b$ of $X$ such that
$a$ is dominated by $b$ and $|A \cap \{a,b\}| = 1$.
We observe that $b$ belongs to~$A$
(since $A$ belongs to $\Dd$ and $a$ is dominated by $b$).
Furthermore, the facts that $a$ belongs to $X$ and $a$ is dominated by $b$
yield that at least $\frac{1}{2}|\Dd|$ sets of $\Dd$
contain $\{a,b\}$ as a subset.
Thus, we can say that $A$ contributes at least $1$ to $w(X)$
(since $A$ contains $b$ but not $a$).
It follows that $w(X) \geq r_4(X) - |\Rr(X)|$,
which concludes the proof of the claim.
\end{proofcl}

As the last claim, we show the following.

\begin{claim} 
\label{c3}
$H$ contains at least four elements $e$ such that $\{e\}$ does not belong to $\Dd$.
\end{claim}

\begin{proofcl}[Proof of Claim~\ref{c3}]
We consider an arbitrary crew $X$.
We note that $r_4(X) \leq 15$
(since $|X| = 6$ and a largest set of $\Dd$ has size~$4$).
However if $r_4(X) = 15$,
then the intersection-closed property implies that $r_2(X) \geq 15$,
and we use the fact that $r_4(X) \geq |\Rr(X)|$
and obtain a contradiction with Claim~\ref{c2}.

Hence, we can assume that $r_4(X) \leq 14$, and thus $r_1(X) \leq 5$ by Claim~\ref{c2}.
In~particular, there is an element $a$ of $X$ 
such that $\{a\}$ does not belong to $\Dd$.
By Claim~\ref{c1}, element $a$ is dominated,
and we let $b$ be an element dominating $a$.
Clearly, $b$ belongs to $H$,
and we consider a crew $X'$ containing
$a$ and $b$.
We use the definition of $\Rr(X')$ and observe that $\Rr(X') \leq 7$.
Consequently,
Claim~\ref{c2} yields that $r_1(X') \leq 2$,
and thus $X'$ contains at most two elements $e$ such that $\{e\}$ belongs to~$\Dd$.
The desired statement follows.
\end{proofcl}

With Claims~\ref{c1},~\ref{c2} and~\ref{c3} on hand, we conclude the proof as follows.
Claim~\ref{c3} gives an element $a$ of $H$ such that the set 
$\{a\}$ does not belong to $\Dd$.
By Claim~\ref{c1},
there is an element $b$ dominating $a$, and we note that $b$ also belongs to $H$.
We consider the set $A$ of all elements of $H$
which are dominated by $b$,
and we discuss two cases based on $|A|$.

For the case that $|A| \geq 3$,
we consider a crew $X$ such that $b$ belongs to $X$ and $|A \cap X| \geq 3$.
We observe that each relevant set contains an element of $A \cap X$
(since $|X| = 6$ and $|A \cap X| \geq 3$),
and hence each relevant set contains $b$, and thus it contains all elements of $A \cap X$.  
The fact that $|(A \cap X) \cup \{b\}| \geq 4$ implies that
$|\Rr(X)| \leq 1$, and we obtain a contradiction with Claim~\ref{c2}.

For the case that $|A| \leq 2$,
we use that $|A \cup \{b\}| \leq 3$ and
Claim~\ref{c3} guarantees that there is an element $a'$ of $H \sm (A \cup \{b\})$ 
such that the set $\{a'\}$ does not belong to~$\Dd$.
In addition, we use Claim~\ref{c1} and consider an element $b'$ dominating $a'$,
and we observe that $b'$ is distinct from $a$ and $b$
(since $a'$ is not dominated  by $b$).
It follows that $a,b,a'$ and $b'$ are four distinct elements of $H$,
and we consider a crew $X$ which contains these four elements.
We note that the only possible relevant sets are
$\{a,b,a',b'\}$ and $X \sm \{a,b\}$ and $X \sm \{a',b'\}$.
We conclude that $|\Rr(X)| \leq 3$ (and if $|\Rr(X)| = 3$ then $r_2(X) \geq 3$),
and a contradiction with Claim~\ref{c2} follows.
\end{proof}


\section{Families with few abundant elements}
\label{families}
In the present section,
we construct various union-closed families of sets 
and we use the families for proving Theorem~\ref{counterexamples}. 

We let $\Pp^8_3$ be the family of sets constructed as follows.
We let $\Aa$ denote the family consisting of all sets $A$ such that
$A \subseteq \{0,1,\dots,7\}$ and
$\{0,1\} \subset A$ and
$|A| \geq 3$.
We let $\Ee = \{ \{0,2,4\}, \{0,2,6\}, \{0,4,6\}, \{0,2,4,6\} \}$
and
$\splitatcommas{ \Oo = \{ \{1,3,5\}, \{1,3,7\}, \{1,5,7\}, \{1,3,5,7\} \} }$,
and we let
$\Pp^8_3 = \Aa \cup \Ee \cup \Oo$.

Also, we extend $\Pp^8_3$ by adding element $8$ to every set,
and we let $\overline{\Pp^9_4}$ denote the resulting family. 

Next, we construct family $\Qq^9_5$ as follows.
We let $\Bb$ denote the family consisting of all sets $B$ such that
$B \subseteq \{0,1,\dots,8\}$ and
$\{0,1,\dots,5\} \subseteq B$.
We let $\Cc$ denote the family consisting of all sets $C$ such that
$C \subseteq \{0,1,2,3,6,7,8\}$ and
$\{0,1,2\} \subset C$ and
$|C| \geq 5$.
We let $\Dd = \{ \{0,1,3,4,5\}, \{0,2,3,4,5\}, \{1,2,3,4,5\} \}$
and
$\Qq^9_5 = \Bb \cup \Cc \cup \Dd$.

Lastly, we construct family $\Rr^{10}_6$.
We let $\Ff$ denote the family consisting of all sets $F$ such that
$F \subseteq \{0,1,\dots,9\}$ and
$\{0,1,\dots,4\} \subset F$ and
$|F| \geq 6$.
For every $i$ of $\{0,1,\dots,4\}$,
we let $G_i$ be the set depicted in Figure~\ref{setsG}
and $\Gg_i$ be the family consisting of $G_i$ 
and all subsets $G'_i$ of $G_i$ such that
$|G'_i| = 6$ and $|G'_i \cap \{0,1,\dots,4\}| = 4$.
For instance,
$\Gg_0 = \{ \{1,2,3,4,7,8\}, \{1,2,3,4,7,9\}, \{1,2,3,4,8,9\}, \{1,2,3,4,7,8,9\} \}$
We let
$\Rr^{10}_6 = \Ff \cup \Gg_0 \cup \Gg_1 \cup \dots \cup \Gg_4$.

\begin{figure}[h!]
    \centering
    \includegraphics[scale=0.65]{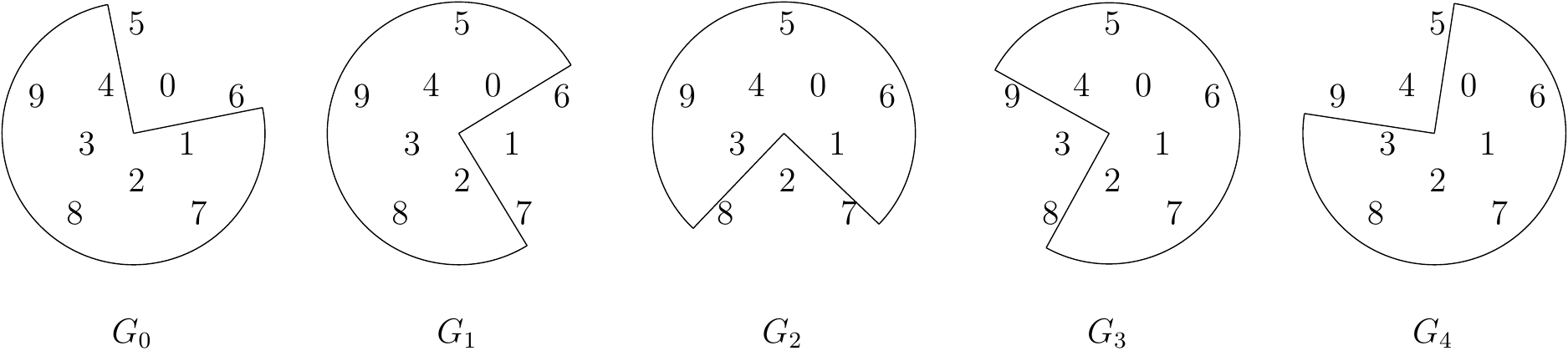}
    \caption{Sets $G_0,G_1,\dots,G_4$ viewed as subsets of $\{0,1,\dots,9\}$.
    For instance, $G_0 = \{1,2,3,4,7,8,9\}$.}
    \label{setsG}
\end{figure}

We show relevant properties of the constructions as follows.

\begin{proposition}
\label{constructions}
Let $\Pp^8_3$, $\overline{\Pp^9_4}$, $\Qq^9_5$ and $\Rr^{10}_6$ be the families of sets defined above.
Each of the families is twin-free and the following statements are satisfied.
\begin{enumerate}
\item
$\Pp^8_3$ is a $(2,3,8)$-construction and $0$ and $1$ are abundant.
\item
$\overline{\Pp^9_4}$ is a $(3,4,9)$-construction and $0, 1$ and $8$ are abundant.
\item
$\Qq^9_5$ is a $(4,5,9)$-construction and $0, 1, 2$ and $3$ are abundant.
\item
$\Rr^{10}_6$ is a $(5,6,10)$-construction and $0, 1, 2, 3$ and $4$ are abundant.
\end{enumerate}
Furthermore, each non-abundant element for $\Pp^8_3$ belongs to strictly less than half of the sets,
and similarly for $\overline{\Pp^9_4}$ and $\Rr^{10}_6$.
\end{proposition}
\begin{proof}
For each of the families $\Pp^8_3$, $\overline{\Pp^9_4}$, $\Qq^9_5$ and $\Rr^{10}_6$,
we note that the subscript and superscript
correspond to the sizes of its smallest and largest set, respectively.
We observe that each of the families is twin-free.
It remains to count the occurrences of elements and show that the families are union-closed.

We show statement (1).
We note that $|\Pp^8_3| = |\Aa| + |\Ee| + |\Oo| = 63 + 4 + 4 = 71$.
Also, we observe that each of elements $0, 1$ belongs to $67$ sets of $\Pp^8_3$
(it belongs to all sets of $\Aa$ and $4$ sets of $\Ee \cup \Oo$), and
each of elements $2,3,\dots,7$ belongs to $35$ sets of $\Pp^8_3$
(it belongs to $32$ sets of $\Aa$ and $3$ sets of $\Ee \cup \Oo$).
Thus, elements $0$ and $1$ are abundant for $\Pp^8_3$
and each of elements $2,3,\dots,7$ belongs to strictly less than half of the sets of $\Pp^8_3$.

In order to show that $\Pp^8_3$ is union-closed,
we consider arbitrary sets $X$ and $Y$ of $\Pp^8_3$
and show that the set $X \cup Y$ belongs to $\Pp^8_3$.
We discuss three cases.
For the case that $1$ does not belong to $X \cup Y$,
we note that $X$ and $Y$ belong to $\Ee$.
The definition of $\Ee$ yields that the set $X \cup Y$ belongs to $\Ee$,
and hence $X \cup Y$ belongs to $\Pp^8_3$.
For the case that $0$ does not belong to $X \cup Y$,
the conclusion follows similarly from the definition of $\Oo$.
Thus, we can assume that $0$ and $1$ belong to $X \cup Y$.
Since $|X| \geq 3$ and $|Y| \geq 3$, we have $|X \cup Y| \geq 3$.
By the definition of $\Aa$, the set $X \cup Y$ belongs to $\Aa$,
and thus to $\Pp^8_3$.

We proceed with statement (2).
We use the properties of $\Pp^8_3$ 
and note that $\overline{\Pp^9_4}$ is also union-closed,  
and each of elements $0, 1$ belongs to $67$ sets of $\overline{\Pp^9_4}$,
and each of elements $2,3,\dots,7$ belongs to $35$ sets of $\overline{\Pp^9_4}$,
and element $8$ belongs to all $71$ sets of $\overline{\Pp^9_4}$.
Thus, elements $0, 1$ and $8$ are abundant for $\overline{\Pp^9_4}$
and each of elements $2,3,\dots,7$ belongs to strictly less than half of the sets of $\overline{\Pp^9_4}$.

Next, we show statement (3).
We note that $|\Qq^9_5| = |\Bb| + |\Cc| + |\Dd| = 8 + 11 + 3 = 22$.
We observe that each of elements $0,1,2$ belongs to $21$ sets, 
and element $3$ belongs to $18$ sets,
and each of elements $4,5,\dots,8$ belongs to $11$ sets
(since no set from $\Cc$ contains $4$ or $5$,
and each of elements $6,7$ and $8$ belongs to
$4$ sets from $\Bb$ and $7$ sets from $\Cc$ and no set from $\Dd$).

We consider arbitrary sets $X$ and $Y$ of $\Qq^9_5$
and discuss two cases.
For the case that both $X$ and $Y$ belong to $\Cc$,
we note that $X \cup Y$ belongs to $\Cc$ (this follows from the definition of $\Cc$),
and thus $X \cup Y$ belongs to $\Qq^9_5$.
For the case that at least one of $X$ and $Y$ belongs to $\Bb \cup \Dd$,
we observe that $\{0,1,\dots,5\}$ is a subset of $X \cup Y$ if $X \neq Y$.
It follows that $X \cup Y$ belongs to $\Bb$,
and thus to $\Qq^9_5$.

Lastly, we show statement (4).
We note that $|\Rr^{10}_6| = |\Ff| + |\Gg_0| + \dots + |\Gg_4| = 31 + 5 \cdot 4 = 51$, and
each of elements $0,1,\dots,4$ belongs to $47$ sets
($31$ sets of $\Ff$ and $16$ sets of $\Gg_0 \cup \dots \cup \Gg_4$), and
each of elements $5,6,\dots,9$ belongs to $25$ sets
($16$ sets of $\Ff$ and $9$ sets of $\Gg_0 \cup \dots \cup \Gg_4$).

We consider arbitrary sets $X$ and $Y$ of $\Rr^{10}_6$.
We use that each set of $\Rr^{10}_6$ contains at least four elements of $\{0,1,\dots,4\}$
and discuss two cases.
For the case that the set $X \cup Y$ contains $\{0,1,\dots,4\}$ as a subset,
we note that $X \cup Y$ belongs to $\Ff$, and thus to $\Rr^{10}_6$.
Otherwise,
we note that $X \cup Y$ contains all but one element of $\{0,1,\dots,4\}$,
and we let $i$ be this element.
It follows that $X$ and $Y$ belong to $\Gg_i$,
and hence $X \cup Y$ belongs to $\Gg_i$,
and thus to $\Rr^{10}_6$.
\end{proof}

With Proposition~\ref{constructions} on hand,
we consider families $\overline{\Pp^9_4}$ and $\Qq^9_5$
and simple extensions of families $\Pp^8_3$ and $\Rr^{10}_6$
and prove Theorem~\ref{counterexamples}.

\begin{proof}[Proof of Theorem~\ref{counterexamples}]
For statement (1), we consider families $\Pp^8_3$, $\overline{\Pp^9_4}$, $\Qq^9_5$ and $\Rr^{10}_6$ and
note that they have the desired properties by Proposition~\ref{constructions}.

In order to show statement (2),
we consider a simple extension $\Pp^n_k$ of the construction of $\Pp^8_3$.
Given integers $n$ and $k$ such that $n \geq k \geq 3$,
we let $\Aa$ be the family consisting of all sets $A$ such that
$A \subseteq \{0,1,\dots,n-1\}$ and
$\{0,1\} \subset A$ and
$|A| \geq k$,
and let $\Ee$ be the family consisting of all sets $E$
such that $E \subseteq \{0,2,\dots, 2\lfloor \frac{n}{2} \rfloor -2\}$
and $0 \in E$ and $|E| \geq k$,
and let $\Oo$ be the family consisting of all sets $O$
such that $O \subseteq \{1,3,\dots, 2\lfloor \frac{n}{2} \rfloor -1\}$ and
$1 \in O$ and $|O| \geq k$,
and we let $\Pp^n_k = \Aa \cup \Ee \cup \Oo$.
In~particular,
$|\Ee| = |\Oo|$ for every $n$ and $k$,
and if $n$ is odd, then element $n-1$ belongs to no set from $\Ee \cup \Oo$.

We observe that $\Pp^n_k$ is twin-free,
and the fact that $\Pp^n_k$ is union-closed
follows by a very similar argument as for the family $\Pp^8_3$.
We also note that elements $0$ and $1$ are abundant for $\Pp^n_k$
since $|\Aa| + |\Ee| > |\Oo|$ and $|\Aa| + |\Oo| > |\Ee|$.

Finally, we use the assumption that $k$ and $n$ satisfy the inequality 
$${\sum_{i=k-1}^{\lfloor \frac{n}{2} \rfloor - 1} \binom{\lfloor \frac{n}{2} \rfloor - 1}{i}
>
\binom{n-3}{k-3} + \binom{\lfloor \frac{n}{2} \rfloor -2}{k-2}},$$
and we show that each of elements $2,3,\dots,n-1$ belongs to strictly less than half of the sets of $\Pp^n_k$.
In particular, we note that each of elements $2,3,\dots,n-2$ belongs to the same number of sets of $\Pp^n_k$,
and element $n-1$ does not belong to a greater number of sets.
Hence, it suffices to consider just element $2$
and show that it belongs to strictly less than half of the sets of $\Pp^n_k$.
We let $\Aa_2$ be the family of all sets of $\Aa$ which contain element $2$
and let $\Aa_0 = \Aa \sm \Aa_2$. Instead of enumerating $|\Aa_2|$ and $|\Aa_0|$,
we enumerate the quantity $a = |\Aa_2| - |\Aa_0|$.
To this end, we consider an arbitrary set $A$ of $\Aa_0$ 
and observe that the set $A \cup \{2\}$ always belongs to $\Aa$, 
and similarly we consider a set $A$ of $\Aa_2$
and observe that the set $A \sm \{2\}$ belongs to $\Aa$ if and only if $|A| > k$.
These two observations imply that $a$ is equal to the number of sets of size $k$ which belong to $\Aa_2$,
that is, $a = \binom{n-3}{k-3}$.
We use a similar reasoning for the sets of $\Ee$.
We let $e$ be equal to the number of sets of $\Ee$ containing element $2$ minus
the number of the remaining sets of $\Ee$,
and we observe that $e = \binom{\lfloor \frac{n}{2} \rfloor -2}{k-2}$.
Lastly, we enumerate the number of sets in $\Oo$.
We use that a smallest set in $\Oo$ is of size $k$
and the largest set is of size $\lfloor \frac{n}{2} \rfloor$ 
and each set contains element $0$, and hence we note that
$$|\Oo| = \sum_{i=k-1}^{\lfloor \frac{n}{2} \rfloor - 1} \binom{\lfloor \frac{n}{2} \rfloor - 1}{i}.$$
We conclude that 
$|\Oo| > a + e$,
and thus element $2$ belongs to strictly less than half of the sets of $\Pp^n_k$.

Finally, we show that   
for every fixed $k$, we can choose $n$ large enough so that 
$${\sum_{i=k-1}^{\lfloor \frac{n}{2} \rfloor - 1} \binom{\lfloor \frac{n}{2} \rfloor - 1}{i}
>
\binom{n-3}{k-3} + \binom{\lfloor \frac{n}{2} \rfloor -2}{k-2}}.$$
In particular if $k-1 \leq \frac{1}{2} \left( \lfloor \frac{n}{2} \rfloor - 1 \right)$,
then it is easy to see that the left-hand side can be bounded by an exponential function as
\begin{equation*}
\sum_{i=k-1}^{\lfloor \frac{n}{2} \rfloor - 1} \binom{\lfloor \frac{n}{2} \rfloor - 1}{i}
\geq
\frac{1}{2}\sum_{i=0}^{\lfloor \frac{n}{2} \rfloor - 1} \binom{\lfloor \frac{n}{2} \rfloor - 1}{i}
=
\frac{1}{2} \cdot 2^{\lfloor \frac{n}{2} \rfloor - 1}
=
2^{\lfloor \frac{n}{2} \rfloor - 2}.
\end{equation*}
For the right-hand side, we note that the first term is bounded by a polynomial as 
\begin{equation*}
\binom{n-3}{k-3}
=
\frac{n-3}{k-3} \cdot \frac{n-4}{k-4} \cdot \dots \cdot \frac{n-k-1}{1}
<
(n-3)^{k-3},
\end{equation*}
and similarly the second term is bounded as
\begin{equation*}
\binom{\lfloor \frac{n}{2} \rfloor -2}{k-2}
<
\left(\lfloor \frac{n}{2} \rfloor -2\right)^{k-2}.
\end{equation*}
For every $k$ and large enough $n$,
we clearly have that 
\begin{equation*} 
2^{\lfloor \frac{n}{2} \rfloor - 2}
>
(n-3)^{k-3} + \left(\lfloor \frac{n}{2} \rfloor -2 \right)^{k-2},
\end{equation*}
and the desired inequality follows.

Next, we show statement (3).
We consider integers $k$ and $n$ which satisfy that
${n \geq \max \{ 3, k + 8\left \lceil \frac{k}{2} \right \rceil - 8 \}}$,
and we note that the term
$k + 8\left \lceil \frac{k}{2} \right \rceil - 8$ is equal to $5k - 8$ for even $k$,
and it is equal to $5k - 4$ for odd $k$ as desired.

For $0 \leq k \leq 3$,
we consider the following simple families and observe that each family is a $(2,k,n)$-construction. 
\begin{itemize}
\item
We consider
$\{ \emptyset, \{0\}, \{1\}, \{0,1\}, \{0,1,\dots ,n-1\} \}$
for $k = 0$.
\item
Similarly,
$\{ \{0\}, \{1\}, \{0,1\}, \{0,1,\dots ,n-1\} \}$
for $k = 1$.
\item
We consider
$\{ \{0,1\}, \{0,1,\dots ,n-1\} \}$
for $k = 2$.
\item
For $k = 3$,
we consider the family obtained from $\Pp^8_3$ 
by adding $n - 8$ new elements, say $8, 9, \dots, n - 1$,
to each set which contains element $2$.
\end{itemize}
Hence,
we can assume that $k \geq 4$ and $n \geq k + 8\left \lceil \frac{k}{2} \right \rceil - 8$,
and we produce a $(2,k,n)$-construction by extending the family $\Pp^{12}_4$.
To this end, we consider sets of additional elements $A_2, A_3,\dots, A_{11}$ such that
$A_2, A_3,\dots, A_{11}$ are pairwise disjoint and each is disjoint with $\{0,1,\dots,11\}$,
and $|A_i| = \left \lceil \frac{k-4}{2} \right \rceil$ for every $i$ of $\{2,3,\dots,9\}$, and
$|A_{10}| = \left \lfloor \frac{k-4}{2} \right \rfloor$ and
$|A_{11}| = n - 12 - \sum_{i = 2}^{10}|A_i|$.
We extend $\Pp^{12}_4$ as follows.
For every $i$ of $\{2,3,\dots,11\}$ in sequence,
we add all elements of $A_i$ to every set containing $i$.
We let $\Pp^+$ denote the resulting family.

Clearly, the largest set of $\Pp^+$ is
$\{0,1,\dots,11\} \cup A_2 \cup A_3 \cup \dots \cup A_{11}$
and its size is
$12 + |A_2| + |A_3| + \dots + |A_{11}|$
which is equal to $n$.
In order to determine the size of a smallest set of $\Pp^+$, we first note the following.
\begin{align*}
|A_{11}|
&= n - 12 - \sum_{i = 2}^{10}|A_i|
= n - 12 - 8 \left \lceil \frac{k-4}{2} \right \rceil - \left \lfloor \frac{k-4}{2} \right \rfloor \\
&= n - k - 7 \left \lceil \frac{k}{2} \right \rceil + 6
\geq k + 8 \left \lceil \frac{k}{2} \right \rceil - 8 - k - 7 \left \lceil \frac{k}{2} \right \rceil + 6 \\
&\geq \left \lceil \frac{k-4}{2} \right \rceil
\end{align*}
Hence, $|A_{11}| \geq |A_i|$ for every $i$ of $\{2,3,\dots,10\}$,
and it follows that a smallest set of~$\Pp^+$ is, for instance, the set $\{0,1,2,10\} \cup A_2 \cup A_{10}$
and its size is $4 + \lceil \frac{k-4}{2} \rceil + \lfloor \frac{k-4}{2} \rfloor$ 
which is equal to $k$.

We use the properties of $\Pp^{12}_4$ and observe that $\Pp^+$ is union-closed.
We conclude that $\Pp^+$ is a $(2,k,n)$-construction,
and $0$ and $1$ are the only abundant elements for $\Pp^+$,
and each non-abundant element belongs to strictly less than half of the sets of~$\Pp^+$.

Lastly, we show statement (4).
We let $k$ and $n$ be arbitrary integers such that
$n - 4 \geq k \geq 3$ and $n \geq 9$
and we produce a $(k - 1,k,n)$-construction.
We discuss three cases.

For the case that $n = 9$ and $k = 5$,
we just consider family $\Qq^9_5$ and use item (3) of Proposition~\ref{constructions}.

For the case that $k = n - 4$ and $n \geq 10$,
we consider family $\Rr^{10}_6$ and its properties given item (4) of Proposition~\ref{constructions},
and we extend the family as follows.
We take $n - 10$ new elements, say $10, 11, \dots, n - 1$, and augment 
every set of $\Rr^{10}_6$ by adding all these elements,
and we let $\overline{\Rr^{n}_k}$ denote the resulting family.
We note that $\overline{\Rr^{n}_k}$ is union-closed 
and its smallest set is of size $k$ and largest set of size $n$
and precisely $k - 1$ elements are abundant for $\overline{\Rr^{n}_k}$ as desired.  
Furthermore,
each of elements $5,6,\dots,9$ belongs to strictly less than half of the sets of $\overline{\Rr^{n}_k}$.

For the case that $n - 5 \geq k \geq 3$ and $n \geq 9$,
we use item (1) of Proposition~\ref{constructions} and we extend family $\Pp^8_3$ as follows.
We choose a non-abundant element, say $2$,
and we augment every set of $\Pp^8_3$ containing $2$ 
by adding elements $8, 9, \dots, n - 1$ (which is $n - 8$ new elements),
and augment every other set by adding elements $8, 9, \dots, k + 4$ (which is $k - 3$ elements),
and we let $\Pp^+$ denote the resulting family.
We observe that $\Pp^+$ is union-closed and
the abundant elements for $\Pp^+$ are $0, 1$ and $8, 9, \dots, k + 4$.
We conclude that $\Pp^+$ is a $(k - 1,k,n)$-construction
and each non-abundant element for $\Pp^+$ belongs to strictly less than half of the sets of $\Pp^+$.
\end{proof}

In relation to statement (3) of Theorem~\ref{counterexamples},
we remark that a simpler $(2,k,n)$-construction can be obtained by extending the family $\Pp^8_3$
(instead of $\Pp^{12}_4$) 
and considering a slightly worse bound of $n \geq \max \{ 3, 6k - 10 \}$.

\section{Relations among conjectures}
\label{relations}

In the present section, we show Proposition~\ref{hierarchy}.
For the sake of completeness,
we also recall the common formulations of the conjectures 
and show that they are equivalent to the formulations stated in Section~\ref{introduction}.
Following~\cite{BS, P, CH},
the formulations are recalled in Conjectures~\ref{franklA}, \ref{poonen3B}, \ref{poonen4C} and~\ref{CuiHuS2D}.

\begin{originalConjecture}
\label{franklA}
If $\Ff$ is a finite union-closed family of sets such that $\Ff \neq \{\emptyset\}$,
then some element belongs to at least half of the sets of $\Ff$.
\end{originalConjecture}

\begin{originalConjecture}
\label{poonen3B}
Let $\Ff$ be a finite union-closed family of sets.
If precisely one element belongs to at least half of the sets of $\Ff$,
then this element belongs to each non-empty set of $\Ff$.
\end{originalConjecture}

\begin{originalConjecture}
\label{poonen4C}
Let $\Ff$ be a finite union-closed family of sets and
let $M$ be the largest set of $\Ff$.
For every pair of elements $a, b$ of $M$, let $\Ff$
contain a set $A$ such that $|A \cap \{a,b\}| = 1$.
Let $x$ be the only element which belongs to at least half of the sets of $\Ff$.
If $|M| \geq 2$, then $\Ff$ consists of the empty set and precisely all sets of $2^M$ containing $x$.
If $|M| = 1$ then $\Ff = \{\{x\}\}$ or $\Ff = \{\emptyset, \{x\}\}$.
\end{originalConjecture}

\begin{originalConjecture}
\label{CuiHuS2D}
If $\Ff$ is a finite union-closed family of sets such that $\Ff \neq \{\emptyset\}$
and a smallest non-empty set of $\Ff$ is of size at least $2$,
then there are at least two elements such that each belongs to a least half of the sets of $\Ff$.
\end{originalConjecture}

We show that the formulations of the conjectures are equivalent as follows.

\begin{proposition}
\label{equivalence}
Each of the following equivalences holds.
\begin{enumerate}
\item
Conjectures~\ref{frankl} and~\ref{franklA} are equivalent.
\item
Conjectures~\ref{poonen3} and~\ref{poonen3B} are equivalent.
\item
Conjectures~\ref{poonen4} and~\ref{poonen4C} are equivalent.
\item
Conjectures~\ref{CuiHuS2} and~\ref{CuiHuS2D} are equivalent.
\end{enumerate}
\end{proposition}

Proving Propositions~\ref{hierarchy} and~\ref{equivalence},
we note that the arguments actually show the desired implications and equivalences
also with the additional constraint that the largest set is of size at most $n$.
For a possible reference in further research,
we prove the relations of Propositions~\ref{hierarchy} and~\ref{equivalence}
in the stronger form with this additional constraint as follows.
\begin{proposition}
\label{relationsAndN}
Let $n$ be a positive integer.
For the sake of brewity,
we say that a conjecture works for $n$
if the conjecture is true under the additional constraint that the largest set of $\Ff$ is of size at most $n$.
The following statements all hold.
\begin{enumerate}
\item
If Conjecture~\ref{poonen4} works for $n$, then Conjecture~\ref{poonen3} works for $n$.
\item
If Conjecture~\ref{poonen3} works for $n$, then Conjecture~\ref{CuiHuS2} works for $n$.
\item
If Conjecture~\ref{CuiHuS2} works for $n$, then Conjecture~\ref{frankl} works for $n$.
\item
Conjecture~\ref{frankl} works for $n$ if and only if Conjecture~\ref{franklA} works for $n$.
\item
Conjecture~\ref{poonen3} works for $n$ if and only if Conjecture~\ref{poonen3B} works for $n$.
\item
Conjecture~\ref{poonen4} works for $n$ if and only if Conjecture~\ref{poonen4C} works for $n$.
\item
Conjecture~\ref{CuiHuS2} works for $n$ if and only if Conjecture~\ref{CuiHuS2D} works for $n$.
\end{enumerate}
\end{proposition}

We observe that it suffices to show Proposition~\ref{relationsAndN}
since Proposition~\ref{relationsAndN} implies Propositions~\ref{hierarchy} and~\ref{equivalence};
for instance, statement (1) of Proposition~\ref{hierarchy} is implied as follows.
Conjecture~\ref{poonen4} is assumed to be true, that is, Conjecture~\ref{poonen4} works for every positive integer $n$,
and we use Proposition~\ref{relationsAndN} and obtain that Conjecture~\ref{poonen3} works for every positive integer $n$,
and thus, Conjecture~\ref{poonen3} is true as desired.

In the remainder of this section, we prove Proposition~\ref{relationsAndN}.

\begin{proof}[Proof of Proposition~\ref{relationsAndN}]
For the sake of brewity,
we let $M(\Ff)$ denote the largest set of a finite union-closed family $\Ff$.
We let $n$ be an arbitrary positive integer, and we fix $n$ and show statements (1), (2), \dots, (7).

We start by showing statement (1), that is, we assume that Conjecture~\ref{poonen4} works for $n$
and show that Conjecture~\ref{poonen3} works for $n$.
In relation to Conjecture~\ref{poonen3},
we let $\Ff_0$ be an arbitrary union-closed family of sets such that
$|M(\Ff_0)| \leq n$ and $\emptyset \not \in \Ff_0$ and
precisely one element is abundant for $\Ff_0$,
and we let $x$ denote this element.
Now, we reduce $\Ff_0$ to a twin-free family as follows. 
If there is a pair of twins in $\Ff_0$, then we choose one of them, say $b$, 
and remove $b$ from every set of $\Ff_0$.
We consider the modified family of sets,
and we apply such removals in sequence until there are no twins in the resulting family,
and we let $\Ff$ denote this family.
We observe that $\Ff$ is a union-closed family of sets and
$|M(\Ff)| \leq n$ and $\emptyset \not \in \Ff$.
Furthermore, the natural bijection between sets of $\Ff_0$ and sets of~$\Ff$
yields that $x$ is the only abundant element for~$\Ff$.
We note that $\Ff$ satisfies the hypothesis of Conjecture~\ref{poonen4}.
Since Conjecture~\ref{poonen4} is assumed to work for~$n$,
we get that $\Ff$ is precisely the family of all sets of $2^{M(\Ff)}$ containing $x$.
It follows that $x$ belongs to every set of the original family $\Ff_0$,
and thus $\Ff_0$ satisfies Conjecture~\ref{poonen3}.
This concludes the proof of statement (1).

In order to show statement (2), we show the contrapositive statement, that is,
we show that if Conjecture~\ref{CuiHuS2} does not work for $n$, then Conjecture~\ref{poonen3} does not work for $n$.
To this end, we consider a family $\Ff$ which witnesses that Conjecture~\ref{CuiHuS2} does not work for $n$.
In particular, family $\Ff$ is union-closed and
its largest set is of size at most $n$,
and smallest set is of size at least $2$,
and at most one element is abundant for $\Ff$.
We discuss two cases.

First,
we suppose that precisely one element is abundant for $\Ff$,
and we let $x$ be this element.
We note that if $x$ does not belong to all sets of $\Ff$, then 
$\Ff$ also witnesses that Conjecture~\ref{poonen3} does not work for $n$, and the desired statement follows.
Hence, we can assume that $x$ belongs to every set of $\Ff$.
We let $S$ be a smallest set of $\Ff$,
and let $\Ff'$ be the family obtained from $\Ff$ by replacing $S$ with the set $S \sm \{x\}$.
We use the fact that $\Ff$ is union-closed and observe that $\Ff'$ is also union-closed
(since for every pair of distinct sets $A, B$ of $\Ff'$,
the set $A \cup B$ contains $x$ and $|A \cup B| > |S|$).
In addition, we note that $|M(\Ff')| \leq n$ and $\emptyset \not \in \Ff'$
and $x$ is the only abundant element for $\Ff'$,
but $x$ does not belong to each set of $\Ff'$.
Thus, family $\Ff'$ witnesses that Conjecture~\ref{poonen3} does not work for $n$.

Second, we suppose that no element is abundant for $\Ff$.
We choose an element $y$ of $M(\Ff)$ and modify $\Ff$ as follows.
For every set $A$ of $\Ff$ such that the set $A \cup \{y\}$ does not belong to $\Ff$,
we remove $A$ from $\Ff$ and add $A \cup \{y\}$ to $\Ff$,
and we let $\Ff^+$ denote the resulting family.
Clearly, $|M(\Ff^+)| \leq n$, and we observe that $\Ff^+$ is union-closed. 
Furthermore, we use the natural bijection between sets of $\Ff$ and sets of $\Ff^+$
and note that $|\Ff| = |\Ff^+|$ and that for every element of $M(\Ff) \sm \{y\}$,
the number of its occurrences in $\Ff$ is equal to the number of occurrences in $\Ff^+$.
We let $c$ be the number of occurrences of $y$ in $\Ff^+$ and we observe that $c \geq \frac{1}{2}|\Ff^+|$.
We discuss three cases based on $c$.

For the case that $c = \frac{1}{2}|\Ff^+|$,
we use the fact that a smallest set of $\Ff^+$ is of size at least $2$
and consider the family $\Ff'$ obtained from $\Ff^+$ by adding the set $\{y\}$.
We observe that $\Ff'$ is union-closed and 
conclude that $\Ff'$ witnesses that Conjecture~\ref{poonen3} does not work for $n$.
 
For the case that $|\Ff^+| > c > \frac{1}{2}|\Ff^+|$,
we note that family $\Ff^+$ witnesses that Conjecture~\ref{poonen3} does not work for $n$.

For the case that $c = |\Ff^+|$,
we let $S$ be a smallest set of $\Ff^+$ and we modify $\Ff^+$ as follows.
We use the fact that $|S| \geq 2$
and let $\Ff'$ be the family obtained from $\Ff^+$ by replacing $S$ with $S \sm \{y\}$.
We observe that $\Ff'$ is union-closed. 
Thus, $\Ff'$ witnesses that Conjecture~\ref{poonen3} does not work for $n$
which concludes the proof of statement (2).

In order to show statement (3),
we let $\Ff$ be a union-closed family of sets such that
$|M(\Ff)| \leq n$ and $\emptyset \not \in \Ff$,
and we discuss two cases.

For the case that a smallest set of $\Ff$ is of size at least $2$,
we use that Conjecture~\ref{CuiHuS2} is assumed to work for $n$
and conclude that some element is abundant for $\Ff$,
and thus $\Ff$ satisfies Conjecture~\ref{frankl}.

For the case that a smallest set of $\Ff$ is of size $1$,
we use a simple idea of~\cite{SR} as follows.
We let $\{x\}$ be a smallest set of $\Ff$
and show that $x$ is abundant for $\Ff$.
We consider a set $A$ of $\Ff$ such that $x$ does not belong to $A$,
and we note that the set $A \cup \{x\}$ belongs to $\Ff$
(since $\Ff$ is union-closed).
For each such set $A$, the set $A \cup \{x\}$ is unique and distinct from $\{x\}$,
and hence the number of sets containing $x$ is greater than the number of sets not containing $x$. 
Thus, $x$ is abundant for $\Ff$,
and $\Ff$ satisfies Conjecture~\ref{frankl}.

For statement (4), we need to show two implications. 
First, we show that if Conjecture~\ref{frankl} works for $n$,
then Conjecture~\ref{franklA} works for $n$.
We let $\Ff$ be a union-closed family of sets such that
$|M(\Ff)| \leq n$ and $\Ff \neq \{\emptyset\}$.
We let $\Ff' = \Ff \sm \emptyset$
and we use that Conjecture~\ref{frankl} is assumed to work for $n$
and note that $\Ff'$ has an abundant element,
and we let $x$ be this element and $c$ be the number of sets of $\Ff'$ containing~$x$.
In particular, $c > \frac{1}{2}|\Ff'|$ and we use the fact that $c$ and $|\Ff'|$ are integers
and obtain $2c \geq |\Ff'| + 1 \geq |\Ff|$.
Hence, $c \geq \frac{1}{2}|\Ff|$ and it follows that $x$ belongs to at least half of the sets of $\Ff$.

Second, we show that if Conjecture~\ref{franklA} works for $n$, then
Conjecture~\ref{frankl} works for~$n$.
We let $\Ff$ be a union-closed family of sets such that
$|M(\Ff)| \leq n$ and $\emptyset \not \in \Ff$.
We let $\Ff' = \Ff \cup \emptyset$ and use that Conjecture~\ref{franklA} is assumed to work for $n$.
We conclude some element belongs to at least half of the sets of $\Ff'$,
and thus this element belongs to more than half of the sets of $\Ff$.

In order to prove statement (5),
we first show that if Conjecture~\ref{poonen3} works for $n$,
then Conjecture~\ref{poonen3B} works for $n$.
To this end, we consider a union-closed family $\Ff$ of sets such that $|M(\Ff)| \leq n$
and such that precisely one element of $M(\Ff)$ belongs to at least half of the sets of $\Ff$;
and we let $x$ be this element.
We let $\Ff' = \Ff \sm \emptyset$ (possibly $\Ff' = \Ff$ if $\emptyset \not \in \Ff$).
We consider an arbitrary element $a$ of $M(\Ff') \sm \{x\}$ and show that $a$ is not abundant for $\Ff'$.
We let $c$ be the number of occurrences of $a$ in $\Ff'$
and note that 
$c < \frac{1}{2}|\Ff|$, and hence $2c \leq |\Ff| - 1 \leq |\Ff'|$,
and thus $c \leq \frac{1}{2}|\Ff'|$ as desired.
On the other hand since Conjecture~\ref{poonen3} is assumed to work for $n$,
we can use statements (2) and (3) of Proposition~\ref{relationsAndN} and obtain that Conjecture~\ref{frankl} works for $n$,
and this yields that $\Ff'$ has at least one abundant element.
It follows that $x$ is the only abundant element for $\Ff'$.
We use that Conjecture~\ref{poonen3} works for $n$
and conclude that $x$ belongs to each set of $\Ff'$,
and thus $x$ belongs to each non-empty set of $\Ff$.

Now, we show that if Conjecture~\ref{poonen3B} works for $n$,
then Conjecture~\ref{poonen3} works for $n$.
We let $\Ff$ be a union-closed family of sets such that $|M(\Ff)| \leq n$ and $\emptyset \not \in \Ff$
and let $x$ be the only abundant element for $\Ff$.
We let $\Ff' = \Ff \cup \emptyset$
and note that $x$ is the only element which belongs to at least half of the sets of $\Ff'$.
We conclude that $x$ belongs of each non-empty set of $\Ff'$,
and thus to every set of $\Ff$.

In order to prove the first implication of statement (6), 
we consider a union-closed family $\Ff$ which satisfies the following.
\begin{itemize}
\item
$|M(\Ff)| \leq n.$
\item
There are no twins in $\Ff$.
\item
Precisely one element of $M(\Ff)$ belongs to at least half of the sets of $\Ff$,
and we let $x$ be this element.
\end{itemize}
Clearly, the desired statement holds if $|M(\Ff)| \leq 1$.
Hence, we can assume that $|M(\Ff)| \geq 2$.
We let $\Ff' = \Ff \sm \emptyset$, and we note that $\Ff'$ has no twins and $|M(\Ff')| \leq n$.
We use a similar argument as in the proof of (5)
and obtain that $x$ is the only abundant element for $\Ff'$.
Since Conjecture~\ref{poonen4} works for $n$,
we get that $\Ff'$ consists of precisely all sets of $2^{M(\Ff')}$ containing $x$.
Finally, we use that $\Ff' = \Ff \sm \emptyset$ and $|M(\Ff)| \geq 2$ 
and that precisely one element of $M(\Ff)$ belongs to at least half of the sets of $\Ff$,
and we conclude that $\Ff$ consists of the empty set and precisely all sets of $2^{M(\Ff)}$ containing $x$.

%

For the reverse implication,
we let $\Ff$ be a union-closed family of sets such that
$|M(\Ff)| \leq n$ and $\emptyset \not \in \Ff$
and such that there are no twins in $\Ff$
and precisely one element is abundant for $\Ff$; and we let $x$ be this element.
Clearly, if $|M(\Ff)| \leq 1$ then $\Ff$ satisfies Conjecture~\ref{poonen4}.
Hence, we can assume that $|M(\Ff)| \geq 2$.
We let $\Ff' = \Ff \cup \emptyset$
and note that $x$ is the only element which belongs to at least half of the sets of $\Ff'$.
Thus, $\Ff'$ consists of the empty set and precisely all sets of $2^{M(\Ff')}$ containing $x$.
It follows that $\Ff$ consists of precisely all sets of $2^{M(\Ff)}$ containing $x$.

Finally, we show statement (7). The proof is similar to the proof of (4).
We let $\Ff$ be a union-closed family of sets 
whose largest set is of size at most $n$ and smallest non-empty set is of size at least $2$.
We let $\Ff' = \Ff \sm \emptyset$ and use that Conjecture~\ref{CuiHuS2} is assumed to work for $n$.
If follows that at least two elements of $M(\Ff')$ belong to more than half of the sets of $\Ff'$,
and thus belong to at least half of the sets of $\Ff$.

For the reverse implication, we let $\Ff$ be a union-closed family of sets
whose largest set is of size at most $n$ and smallest set is of size at least $2$.
We consider the family $\Ff \cup \emptyset$ and use a similar argument,
and we conclude that at least two elements are abundant for $\Ff$.
\end{proof}


\section*{Acknowledgement}
We thank Stijn Cambie for his suggestion on extending the constructions $\Pp^8_3$ and $\Pp^{12}_4$ to $\Pp^n_k$.



\begin{thebibliography}{99}
%
\bibitem{AHS}
R. Alweiss, B. Huang and M. Sellke:
Improved Lower Bound for the Union-Closed Sets Conjecture,
arXiv:2211.11731 (2022).

\bibitem{BBE}
I. Balla, B. Bollob\'{a}s and T. Eccles:
Union-closed families of sets,
Journal of Combinatorial Theory, Series A 120 (2013), 531--544.
%
\bibitem{BM}
I. Bo\v{s}njak and P. Markovi\'{c}:
The $11$-element case of Frankl's conjecture,
European Journal of Combinatorics 15 (2008), R88.

\bibitem{BS}
H. Bruhn and O. Schaudt:
The journey of the union-closed sets conjecture,
Graphs and Combinatorics 31 (2015), 2043--2074.

\bibitem{BCST}
H. Bruhn, P. Charbit, O. Schaudt and J. A. Telle:
The graph formulation of the union-closed sets conjecture,
European Journal of Combinatorics 43 (2015), 210--219.

\bibitem{C}
S. Cambie:
Better bounds for the union-closed sets conjecture using the entropy approach,
arXiv:2212.12500 (2022).

\bibitem{CL}
Z. Chase and S. Lovett:
Approximate union closed conjecture,
arXiv:2211.11689 (2022).

\bibitem{CH}
Z. Cui and Z. Hu:
Two stronger versions of the union-closed sets conjecture,
arXiv:1711.04276 (2019).

\bibitem{LF}
G. Lo Faro:
Union-closed sets conjecture: improved bounds,
Journal of Combinatorial Mathematics and Combinatorial Computing 16 (1994), 97--102.

\bibitem{G}
J. Gilmer:
A constant lower bound for the union-closed sets conjecture,
arXiv:2211.09055 (2022).

\bibitem{F}
P. Frankl:
Extremal set systems, Handbook of combinatorics,
Elsevier (1995), 1293--1329.

\bibitem{Ma}
P. Markovi\'{c}:
An attempt at Frankl's conjecture,
Publications de l'Institut Math\'{e}matique 81 (2007), 29--43.

\bibitem{Mo}
R. Morris:
FC-families and improved bounds for Frankl's conjecture,
European Journal of Combinatorics 27 (2006), 269--282.

\bibitem{Pe}
L. Pebody:
Extension of a Method of Gilmer, 
arXiv:2211.13139 (2022).

\bibitem{P}
B. Poonen:
Union-closed families,
Journal of Combinatorial Theory, Series A 59 (1992), 253--268.

\bibitem{S}
W. Sawin:
An improved lower bound for the union-closed set conjecture,
arXiv:2211.11504 (2022).

\bibitem{SR}
D. G. Sarvate and J.-C. Renaud:
On the union-closed sets conjecture,
Ars Combinatorica 27 (1989), 149--154.

\bibitem{VZ}
B. Vu\v{c}kovi\'{c} and M. \v{Z}ivkovi\'{c}:
The $12$-element case of Frankl's conjecture,
IPSI BgD Transactions on Internet Research 13 (2017), 65--71.
%
\bibitem{W}
P. W\'{o}jcik:
Union-closed families of sets, Discrete Mathematics 199 (1999), 173--182.
%
\end{thebibliography}
\end{document}